\documentclass{amsart}
\usepackage{amsfonts}
\usepackage{amssymb, amsmath, amsthm, graphics, comment, xspace, enumerate}

\newtheorem{thm}{Theorem}[section]

\newtheorem{prop}[thm]{Proposition}
\theoremstyle{definition}
\newtheorem{defn}[thm]{Definition}
\newtheorem{example}[thm]{Example}
\theoremstyle{remark}
\newtheorem{rem}[thm]{Remark}
\numberwithin{equation}{section}

\input{tcilatex}

\begin{document}
\title[Semi-Bloch periodic functions, semi-anti-periodic functions...]{Semi-Bloch periodic functions, semi-anti-periodic functions and applications}
\author{Belkacem Chaouchi}
\address{Lab. de l'Energie et des Syst\' emes Intelligents, Khemis Miliana
University, 44225 Khemis Miliana, Algeria}
\email{chaouchicukm@gmail.com}
\author{Marko Kosti\' c}
\address{Faculty of Technical Sciences, University of Novi Sad, Trg D.
Obradovi\' ca 6, 21125 Novi Sad, Serbia}
\email{marco.s@verat.net}
\author{Stevan Pilipovi\'c}
\address{Department for Mathematics and Informatics, University of Novi Sad,
Trg D. Obradovi\' ca 4, 21000 Novi Sad, Serbia}
\email{pilipovic@dmi.uns.ac.rs}
\author{Daniel Velinov}
\address{Faculty of Civil Engineering, Ss. Cyril and Methodius University,
Skopje, Partizanski Odredi 24, P.O. box 560, 1000 Skopje}
\email{velinovd@gf.ukim.edu.mk}

\begin{abstract}
In this paper, we introduce the notions of semi-Bloch periodic functions and
semi-anti-periodic functions. Stepanov semi-Bloch periodic functions and
Stepanov semi-anti-periodic functions are considered, as well. We analyze the invariance of introduced classes under the actions of convolution products
and briefly explain how one can use the obtained results in the qualitative analysis of solutions of
abstract inhomogeneous integro-differential equations.
\end{abstract}

\maketitle

{\renewcommand{\thefootnote}{} \footnote{%
\text{ } \ \ \textit{Key words and phrases.} Semi-Bloch $k$-periodic functions, semi-anti-periodic functions, integro-differential equations. \newline
\text{ } \ \ This research is partially supported by grant 174024 of
Ministry of Science and Technological Development, Republic of Serbia.}}

\section{Introduction and preliminaries}

Let $p>0$ and $k\in {\mathbb R}.$ The study of Bloch $(p,k)$-periodic functions is an important subject of applied
functional analysis. In fact, this kind of functions is often encountered in
quantum mechanics and solid state physics. Also, they are very attractive
topic in the qualitative theory of differential equations (see the papers \cite{has2} by  M. F. Hasler, \cite{has1} by M. F. Hasler, G.M.N'. Gu\'{e}r\'{e}kata, \cite{kostic2} by M. Kosti\' c, D. Velinov and references cited therein). As it is well known, 
the notion of an anti-periodic function is a special case of the notion of a Bloch $(p,k)$-periodic function. For more details about anti-periodic type functions and their applications, we 
refer the reader to \cite{Ch, dimba, guer2, kosticdaniel, Li, Liu} and references cited therein.

The genesis of this paper is motivated by reading the paper \cite{andres1} by J. Andres and D. Pennequin, 
where the authors have introduced and analyzed the class of semi-periodic functions (sequences) and related 
applications to differential (difference) equations; see also \cite{andres2}. For the sake of brevity, in this paper
we will consider only functions and differential equations. 

The main ideas and organization of paper can be briefly described as follows. After collecting necessary definitions and results needed for our further work, in Section 2
we introduce and study the classes of semi-Bloch $k$-periodic functions and
semi-anti-periodic functions (cf. Definition \ref{saw} and Definition \ref{semi-anti}). Any semi-anti-periodic function is semi-periodic and almost anti-periodic (\cite{kosticdaniel}), while any anti-periodic function is anti-semi-periodic (cf. Proposition \ref{zajebano}). In actual fact, the space consisting of all semi-anti-periodic 
functions is nothing else but the closure of space of all anti-periodic functions in the space of bounded continuous functions (cf. Proposition \ref{novine}). Similarly, the space consisting of all 
semi-Bloch $k$-periodic functions is, in fact, the closure of the space consisting of all Bloch $(p,k)$-periodic functions, where the parameter $p>0$ runs through the positive real axis (cf. Proposition \ref{okje}).
In the same proposition, we show that a function $f(\cdot)$ is semi-Bloch $k$-periodic iff the function $e^{-ik\cdot}f(\cdot)$ is semi-periodic; with $k=0,$ this space is therefore reduced to the space of semi-periodic functions.  
We provide several new observations about Bloch $(p,k)$-periodic functions; for example, any Bloch $(p,k)$-periodic function needs to be almost periodic; see also Remark \ref{fela}.
Finally, in Definition \ref{sqaz} we introduce the notions of asymptotical semi-periodicity (asymptotical semi Bloch $k$-periodicity, asymptotical semi-anti-periodicity). 

In \cite{andres1}, the authors have considered only continuous functions. The main aim of Section 3 is to introduce and analyze the classes of Stepanov semi-Bloch $k$-periodic functions and
Stepanov semi-anti-periodic functions. The invariace of introduced classes under the actions of convolution products
are considered in Section 4, where it is shortly explained, without giving full details, how one can use the obtained results in the qualitative analysis of solutions of
abstract inhomogeneous integro-differential equations and inclusions. In addition to the above, we present several illustrative examples of Stepanov semi-Bloch $k$-periodic functions and Stepanov semi-periodic-functions in Section 2 and Section 3.

We use the standard terminology throughout the paper. Unless stated otherwise, we will always assume henceforth that $(E,\| \cdot \|)$ and $(X,\| \cdot \|_{X})$ are two complex Banach spaces. By $L(E,X)$ we denote the space consisting of all bounded continuous mappings from $E$ into $X;$ $L(E)\equiv L(E,E).$
Assume that $I={\mathbb R}$ or $I=[0,\infty).$ By $C_{b}(I:E)$ we denote the space consisting of all bounded continuous functions from $I$ into $E;$ $C_{0}([0,\infty):E)$ denotes the closed subspace of $C_{b}(I:E)$ consisting of all functions vanishing at infinity.
By $BUC(I:E)$ we denote the space consisted of all bounded uniformly continuous functions $f: I \rightarrow E.$
Finally, by $P_{c}(I :E)$ we denote the space consisting of all continuous periodic functions $f: I\rightarrow E.$ The sup-norm turns these spaces into Banach's.
We will use the symbol $\ast$ to denote the infinite convolution product; generally, if the functions $a :{\mathbb R}\rightarrow {\mathbb C}$ and
$f : {\mathbb R}\rightarrow E$ are measurable, then we denote $(a\ast f)(x):=\int_{-\infty}^{+\infty}a(x-y)f(y)\, dy,$ $x\in {\mathbb R}$, if this integral is well-defined. 

The concept of almost periodicity was introduced by Danish mathematician H. Bohr around 1924-1926 and later generalized by many other authors. Suppose that a function $f : I \rightarrow E$ is continuous. Given $\epsilon>0,$ we call $\tau>0$ an $\epsilon$-period for $f(\cdot)$ iff
\begin{align*}
\| f(t+\tau)-f(t) \| \leq \epsilon,\quad t\in I.
\end{align*}
The set constituted of all $\epsilon$-periods for $f(\cdot)$ is denoted by $\vartheta(f,\epsilon).$ It is said that $f(\cdot)$ is almost periodic iff for each $\epsilon>0$ the set $\vartheta(f,\epsilon)$ is relatively dense in $[0,\infty),$ which means that
there exists $l>0$ such that any subinterval of $[0,\infty)$ of length $l$ meets $\vartheta(f,\epsilon)$.
The space consisted of all almost periodic functions from the interval $I$ into $E$ will be denoted by $AP(I:E).$ Equipped with the sup-norm, $AP(I:X)$ becomes a Banach space. Bohr's transform of any function $f\in AP(I:X),$ defined by
$$
P_{r}(f) := \lim_{t\rightarrow \infty}\frac{1}{t}\int^{t}_{0}e^{-irs}f(s)\, ds,
$$
exists for all $r\in {\mathbb R}.$ The element $
P_{r}(f)$ is called the Bohr's coefficient of $f(\cdot)$ and we know that, if $P_{r}(f) = 0$ for all $r \in {\mathbb R},$ then $f(t) = 0$ for all $t \in I$ as well as that
Bohr's spectrum $\sigma(f):=\{r\in {\mathbb R} :  P_{r}(f) \neq 0\}$ is at most countable. 
By $AP_{\Lambda}(I:X),$ where $ \Lambda$ is a non-empty subset of $I,$ we denote the vector subspace of $AP(I:X)$ consisting of all functions $f\in AP(I:X)$ for which the inclusion $\sigma(f)\subseteq \Lambda$ holds good.

For the sequel, we need some preliminary results from the pioneering paper \cite{Ba} by H. Bart and S. Goldberg, who introduced the notion of an almost periodic strongly continuous semigroup there.
The translation semigroup $(W(t))_{t\geq 0}$ on $AP([0,\infty) : E),$ given by $[W(t)f](s):=f(t+s),$ $t\geq 0,$ $s\geq 0,$ $f\in AP([0,\infty) : E)$ is consisted solely of surjective isometries $W(t)$ ($t\geq 0$) and can be extended to a $C_{0}$-group $(W(t))_{t\in {\mathbb R}}$ of isometries on $AP([0,\infty) : E),$ where $W(-t):=W(t)^{-1}$ for $t>0.$ Furthermore, the mapping ${\mathbb E} : AP([0,\infty) : E) \rightarrow AP({\mathbb R} : E),$ defined by
$$
[{\mathbb E} f](t):=[W(t)f](0),\quad t\in {\mathbb R},\ f\in AP([0,\infty) : E),
$$
is a linear surjective isometry and ${\mathbb E}f(\cdot)$ is the unique continuous almost periodic extension of function $f(\cdot)$ from $AP([0,\infty) : E)$ to the whole real line. We have that $
[{\mathbb E}(Bf)]=B({\mathbb E}f)$ for all $B\in L(E)$ and $f\in  AP([0,\infty) : E).$

Let $1\leq q <\infty.$ A function $f\in L^{q}_{loc}(I :E)$ is Stepanov $q$-bounded iff\index{function!Stepanov bounded}
$$
\|f\|_{S^{q}}:=\sup_{t\in I}\Biggl( \int^{t+1}_{t}\|f(s)\|^{q}\, ds\Biggr)^{1/q}<\infty.
$$
Equipped with the above norm, the space $L_{S}^{q}(I:E)$ consisted of all $S^{q}$-bounded functions is the Banach space.
A function $f\in L_{S}^{q}(I:E)$ is said to be Stepanov $q$-almost periodic, $S^{q}$-almost periodic shortly, iff the function
$
\hat{f} : I \rightarrow L^{q}([0,1] :E),
$ defined by
\begin{align}\label{bruka}
\hat{f}(t)(s):=f(t+s),\quad t\in I,\ s\in [0,1],
\end{align}
is almost periodic .
It is said that $f\in  L_{S}^{q}([0,\infty): E)$ is asymptotically Stepanov $q$-almost periodic iff the function $\hat{f} : [0,\infty) \rightarrow L^{q}([0,1]:E)$ is asymptotically almost periodic.

Further on, 
let us recall that $f(\cdot)$ is anti-periodic iff there exists $p>0$ such that $f(x+p)=-f(x),$ $x\in I.$ Any such function needs to be periodic, as it can be easily proved.
Given $\epsilon>0,$ we call $\tau>0$ an $\epsilon$-antiperiod for $f(\cdot)$ iff
\begin{align*}
\| f(t+\tau) +f(t) \| \leq \epsilon,\quad t\in I.
\end{align*}
By $\vartheta_{ap}(f,\epsilon)$ we denote the set of all $\epsilon$-antiperiods for $f(\cdot).$
The notion of an almost anti-periodic function has been recently introduced in \cite[Definition 2.1]{kosticdaniel} as follows:

\begin{defn}\label{anti-periodic}
It is said that $f(\cdot)$ is almost anti-periodic iff for each $\epsilon>0$ the set $\vartheta_{ap}(f,\epsilon)$ is relatively dense in $[0,\infty).$
\end{defn}

We know that any almost anti-periodic function needs to be almost periodic. Denote by $ANP_{0}(I : E)$ the linear span of almost anti-periodic functions from $I$ into $X.$ Then \cite[Theorem 2.3]{kosticdaniel} implies that $ANP_{0}(I : E)$ is a linear subspace of $AP(I : E)$ as well as that the linear closure of $ANP_{0}(I : E)$ in $AP(I : E),$
denoted by 
$ANP(I : E),$ satisfies
\begin{align}\label{supa}
ANP(I : E)=AP_{{\mathbb R} \setminus \{0\}}(I : E).
\end{align}

Let $1\leq q<\infty.$ We refer the reader to \cite[Definition 3.1]{kosticdaniel} for the notions of an asymptotically
almost anti-periodic function $f :[0,\infty) \rightarrow E$ and an asymptotically Stepanov $q$-almost anti-periodic function $f :[0,\infty) \rightarrow E.$

For more details about almost periodic and almost automorphic type functions, we refer the reader to the monographs \cite{besi} by A. S. Besicovitch, \cite{Di} by T. Diagana, \cite{guer1} by G.M.N'. Gu\'{e}r\'{e}kata and \cite{kostic1} by M. Kosti\' c.

\section{Semi-Bloch $k$-periodicity}\label{prvasekc}

We will always assume that $I=[0,\infty )$ or $I={\mathbb{R}};$
set ${\mathbb S}:={\mathbb N}$ if $I=[0,\infty),$ and
${\mathbb S}:={\mathbb Z}$ if $I={\mathbb R}.$

For the convenience of the reader, we must recall that a bounded continuous
function $f:I\rightarrow E$ is said to be Bloch $(p,k)$-periodic, or Bloch
periodic with period $p$ and Bloch wave vector or Floquet exponent $k$ iff 
$
f(x+p)=e^{ikp}f(x),$ $ x\in I,
$
with $p>0$ and $k\in {\mathbb{R}}$.
The space of all functions $f:I\rightarrow E$ that are Bloch $(p,k)$-periodic will be denoted by ${\mathcal{B}}_{p,k}(I:E)$. 
If $f\in {\mathcal{B}}_{p,k}(I:E),$ then we have
\begin{align*}
f(x+mp)=e^{ikmp}f(x),\quad x\in I,\ m\in {\mathbb S}.
\end{align*}
Given $k\in {\mathbb{R}},$ we 
set ${\mathcal{B}}_{k}(I:E):=\bigcup_{p> 0}{\mathcal{B}}_{p,k}(I:E).$
Observing that $f\in P_{c}(I:X)$ satisfies $f(x+p)=f(x)$ for all $x\in I$ and some $p>0$ iff the function
$F(x):=e^{ikx}f(x),$ $x\in I$ satisfies $F(x+p)=e^{ikp}F(x),$ $x\in I$, we may conclude that
\begin{align}\label{droq}
{\mathcal{B}}_{k}(I:E):=\bigl\{e^{ik\cdot} f(\cdot) : f\in P_{c}(I:E)\bigr\}.
\end{align}

For more
details on the Bloch $(p,k)$-periodic functions, see \cite{has1}, \cite{kostic2} and references cited therein. 

Let us define the notion of a semi-Bloch $k$-periodic function as follows:

\begin{defn}\label{saw}
Let $k\in {\mathbb{R}}.$
A function $f\in C_{b}(I:E)$ is said to be
semi-Bloch $k$-periodic iff 
\begin{equation}\label{dar}
\forall \varepsilon >0\quad \exists p> 0\quad \forall m\in {\mathbb{S}}%
\quad \forall x\in I \quad \Vert f(x+mp)-e^{ikmp}f(x)\Vert \leq
\varepsilon .
\end{equation}
The space of all semi-Bloch $k$-periodic functions will be denoted by 
$
\mathcal{S}B_{k}(I:E)$.
\end{defn}

It is clear that Definition \ref{saw} provides a generalization of \cite[Definition 2 and Definition 3]{andres1}, given only in the case that $I={\mathbb R}.$ Speaking-matter-of-factly, a function $f: {\mathbb R} \rightarrow E$ is semi-periodic in the sense of above-mentioned (equivalent) definitions iff $f: {\mathbb R} \rightarrow E$ is semi-Bloch $0$-periodic. Further on,
it can be easily seen that for each $k\in {\mathbb R}$ any constant function $f\equiv c$ belongs to the space $
\mathcal{S}B_{k}(I:E);$ for this, it is only worth noticing that for each $\epsilon>0$ and $k\neq 0$ we can take $p=2\pi/k$ and \eqref{dar} will be satisfied.

\begin{rem}\label{fela}
It is not so easy to introdude the concept of almost Bloch $k$-periodicity, where $k\in {\mathbb R}.$ In order to explain this in more detail,
assume that a function $f\in C_{b}(I:E)$ and a number $\varepsilon >0$ are given. Let us say that a real number $p> 0$ is an $(\varepsilon,k) $-Bloch period
for $f(\cdot )$ iff 
\begin{align}\label{pop}
\Bigl\| f(x+p)-e^{ikp}f(x)\Bigr\|\leq \varepsilon ,\quad x\in I,
\end{align}
and $f(\cdot )$ is almost Bloch $k$-periodic iff for each $%
\varepsilon >0$ 
the set constituted of all $(\varepsilon,k) $-Bloch periods
for $f(\cdot )$ is relatively dense in $[0,\infty).$ But, then we have $f(\cdot)$ is almost Bloch $k$-periodic 
iff $f(\cdot)$ is almost periodic. To see this, it suffices to observe that
\eqref{pop} is equivalent with
$$
\Bigl\| e^{-ik(x+p)}f(x+p) -e^{-ik x}f(x)\Bigr\|\leq \epsilon,\quad x\in I,
$$
so that, actually, the function $f(\cdot)$ is almost Bloch $k$-periodic iff the function $e^{-ik\cdot}f(\cdot)$ is almost periodic, which is equivalent to say that 
the function $f(\cdot)$ is almost periodic due to \cite[Theorem 2.1.1(ii)]{kostic1}. Further on, let  $f(\cdot) \in{\mathcal{S}B}_{k}(I
:E).$ Then for each number $\epsilon>0$ we have that the set constituted of all $(\varepsilon,k) $-Bloch periods
for $f(\cdot )$ is relatively dense in $[0,\infty)$ since it contains the set $\{mp : m\in {\mathbb N}\},$ where $p>0$ is determined by \eqref{dar}. In view of our previous conclusions, we get that $f(\cdot)$ is almost periodic. In particular, any Bloch $(p,k)$-periodic function needs to be almost periodic, which has not been observed in the researches of Bloch periodic functions carried out so far (see, e.g., \cite{dimba}-\cite{has1} and \cite{kostic2}). 
\end{rem}

Now we will prove the following fundamental result:

\begin{prop}\label{okje}
Let $k\in {\mathbb R}$ and $f\in C_{b}(I:E).$ Then the following holds:
\begin{itemize}
\item[(i)] $f(\cdot)$ is semi-Bloch $k$-periodic iff  $e^{-ik\cdot}f(\cdot)$ is semi-periodic.  
\item[(ii)] $f(\cdot)$ is semi-Bloch $k$-periodic iff there exists a sequence $(f_{n})$ in $P_{c}(I :E)$ such that $\lim_{n\rightarrow \infty}e^{ikx}f_{n}(x)=f(x)$ uniformly in $I.$
\item[(iii)] $f(\cdot)$ is semi-Bloch $k$-periodic iff there exists a sequence $(f_{n})$ in ${\mathcal B}_{k}(I:E)$ such that $\lim_{n\rightarrow \infty}f_{n}(x)=f(x)$ uniformly in $I.$
\end{itemize}
\end{prop}

\begin{proof}
The proof of (i) follows similarly as in Remark \ref{fela}. Since \cite[Lemma 1 and Theorem 1]{andres1} hold for the functions defined on the interval $I=[0,\infty),$ we have that (i) implies that $f(\cdot)$ is semi-Bloch $k$-periodic iff there exists a sequence $(f_{n})$ in $P_{c}(I:E)$ such that
$\lim_{n\rightarrow \infty}e^{ikx}f_{n}(x)=f(x)$ uniformly in $I.$ This proves (ii). For the proof of (iii), it suffices to apply (ii), \eqref{droq} and the conclusion preceding it.
\end{proof}

Let $k\in {\mathbb R}$. Using Proposition \ref{okje} and \cite[Proposition 2]{andres1}, we may construct a substantially large class of semi-Bloch $k$-periodic functions, which do not form a vector space due to a simple example in 
the second part of \cite[Remark 3]{andres1};
\cite[Lemma 2]{andres1} can be straightforwardly reformulated for semi-Bloch $k$-periodic functions, while the function given in \cite[Example 1]{andres1} can be simply used to provide an example of a scalar-valued semi-Bloch $k$-periodic function
which is not contained in the space ${\mathcal B}_{k}(I: {\mathbb C}).$ If we define Bloch $k$-quasi periodic function
\begin{align}
{\mathcal{B}}_{k;q}(I:E):=\bigl\{e^{ik\cdot} f(\cdot) : f\in QP^{0}(I:E)\bigr\},
\end{align}
where $QP^{0}(I:E)$ denotes the space of all quasi-periodic functions from $I$ into $E$ (see \cite{andres1} and references cited therein for the notion),
then
\cite[Theorem 2]{andres1} can be also reformulated in our context; this also holds for \cite[Example 2, Example 3]{andres1}.

By the foregoing, we have:
\begin{equation*}
{\mathcal{B}}_{k}(I :E)\subseteq {\mathcal{S}B}_{k}(I
:E)\subseteq AP(I:E)\subseteq BUC(I:E),\quad  k\in {\mathbb R} .
\end{equation*}

\begin{example}\label{kosinus}
The function $f(x):=\cos x,$ $x\in {\mathbb R}$ is anti-periodic. Now we will prove that $f\in{\mathcal{S}B}_{k}(I
:E)$ iff $k\in {\mathbb Q}.$ For $k\in {\mathbb Q},$ this is clear because we can take $p$ in \eqref{dar} as a certain multiple of $2\pi$. Let us assume now that 
$k\notin {\mathbb Q}.$ Then it suffices to show that the function $e^{-ik\cdot}f(\cdot)$ is not semi-periodic. Towards see this, let us observe that $\sigma(e^{-ik\cdot}f(\cdot))=\{1-k,-1-k\}$ so that there does not exist a positive real number $\theta>0$ 
such that $\sigma(e^{-ik\cdot}f(\cdot)) \subseteq \theta \cdot {\mathbb Q},$ which can be simply approved and which contradicts \cite[Lemma 2]{andres1}.
\end{example}

\begin{rem}\label{a-ova}
Let $a\in AP(I:{\mathbb C}).$ 
Based on Proposition \ref{okje}, we can introduce and analyze the following notion: A function $f\in C_{b}(I:E)$ is said to be semi-$a$-periodic iff there exists a sequence $(f_{n})$ in $P_{c}(I :E)$ such that $\lim_{n\rightarrow \infty}a(x)f_{n}(x)=f(x)$ uniformly in $I.$
Any such function needs to be almost periodic due to \cite[Theorem 2.1.1(ii)]{kostic1}. We will analyze this notion somewhere else.
\end{rem}

Further on, if $f : I \rightarrow E$ is anti-periodic and $f(x+p)=-f(x),$ $x\in I$ for some $p>0,$  
then we have $f(x+mp)=(-1)^{m}f(x),$ $x\in I$, $m\in {\mathbb S}.$ Therefore, it is meaningful to consider the following notion:

\begin{defn}\label{semi-anti}
Let $f\in C_{b}(I:E).$
Then we say that $f(\cdot )$ is semi-anti-periodic iff
\begin{equation*}
\forall \varepsilon >0\quad \exists p> 0\quad \forall m\in {\mathbb{S}}%
\quad \forall x\in I \quad \Vert f(x+mp)-(-1)^{m}f(x)\Vert \leq
\varepsilon .
\end{equation*}
The space of all semi-anti-periodic functions will be denoted by 
$
\mathcal{SANP}(I:E)$.
\end{defn}
It immediately follows from the foregoing that any anti-periodic function is semi-anti-periodic. Furthermore, 
any semi-anti-periodic function $f : I \rightarrow E$ is both semi-periodic and almost anti-periodic. To see this, let $\epsilon>0$ be given in advance. Then we can find $p>0$ such that
\eqref{dar} holds with $\epsilon$ replaced by $\epsilon/2$ therein. Then \eqref{dar} holds with $p$ replaced therein with $2p$ and $k=0,$ which follows from the simple estimates
\begin{align*}
\|f(x+2mp)-f(x)\| &\leq \| f(x+2mp)-(-1)^{m}f(x+mp) \| 
\\&+\|(-1)^{m}f(x+mp)-f(x)\|<\epsilon,\quad x\in I,\ m\in {\mathbb S}.
\end{align*}
Hence, $f(\cdot)$ is semi-periodic. To see that $f(\cdot)$ is almost anti-periodic, notice only that \eqref{dar} implies 
$$
\|f(x+(2m+1)p)+f(x)\| \leq \epsilon,\quad x\in I,\ m\in {\mathbb S}
$$
as well as the set $\{(2m+1)p : m\in {\mathbb N}\}$ is relatively dense in $[0,\infty).$ Therefore, we have proved the following

\begin{prop}\label{zajebano}
Let $f\in C_{b}(I:E).$ 
\begin{itemize}
\item[(i)] If $f(\cdot)$ is anti-periodic, then $f(\cdot)$ is semi-anti-periodic.
\item[(ii)] If $f(\cdot)$ is semi-anti-periodic, then $f(\cdot)$ is semi-periodic and almost anti-periodic.
\end{itemize}
\end{prop}

Furthermore, applying the same arguments as in the proofs of \cite[Lemma 1 and
 Theorem 1]{andres1}, we may conclude that the following holds:

\begin{prop}\label{novine}
Let $f\in C_{b}(I :E).$ Then $f(\cdot)$ is semi-anti-periodic iff there exists a sequence $(f_{n})$ of anti-periodic functions in $C_{b}(I :E)$ such that $\lim_{n\rightarrow \infty}f_{n}(x)=f(x)$ uniformly in $I.$   
\end{prop}

Let $p>0$ and $k\in {\mathbb R}.$ If $f : {\mathbb R} \rightarrow E$ is periodic (anti-periodic, Bloch $(p,k)$-periodic) and $a\in L^{1}({\mathbb R}),$ then the function $a\ast f(\cdot)$ is likewise periodic (anti-periodic, Bloch $(p,k)$-periodic). Using the Young inequality, Proposition \ref{okje} and Proposition \ref{novine}, it immediately follows that the 
space of semi-Bloch $k$-periodic functions and the space of semi-anti-periodic functions are convolution invariant. 

We continue by providing several illustrative examples:

\begin{example}\label{demos}
Let $f\equiv c\neq 0.$ Due to \eqref{supa}, $
f\notin ANP({\mathbb R} : E)$ and clearly $f(\cdot)$ is not semi-anti-periodic. On the other hand, $f(\cdot)$ is periodic and therefore semi-periodic.
\end{example}

\begin{example}\label{strina}
It can be simply verified that the function $f(x):=\sin x +\sin \pi x\sqrt{2},$ $x\in {\mathbb R}$ is almost anti-periodic but not semi-periodic (see, e.g., \cite[Remark 3]{andres1} and \cite[Example 2.1]{kosticdaniel}).
\end{example}

\begin{example}\label{strina1} (a slight modification of \cite[Example 1]{andres1})
The function
$$
f(x):=\sum_{n=1}^{\infty}\frac{e^{ix/(2n+1)}}{n^{2}},\quad x\in {\mathbb R} 
$$
is semi-anti-periodic because it is a uniform limit of $[\pi \cdot (2n+1)!!]$-anti-periodic functions
$$
f_{N}(x):=\sum_{n=1}^{\infty}\frac{e^{ix/(2n+1)}}{n^{2}},\quad x\in {\mathbb R} \ \ (N\in {\mathbb N}).
$$
On the other hand, the function $f(\cdot)$ cannot be periodic. 
\end{example}

\begin{example}\label{olomuc}
Set ${\mathbb Q}_{n}:=\{(2n+1)/(2m+1) : m,\ n\in {\mathbb Z}\}.$ If $\theta>0$ and $\sum_{\lambda \in \theta \cdot {\mathbb Q}_{n}}\|a_{\lambda}(f)\|<\infty,$ then the function
$$
f(t):=\sum_{\lambda \in \theta \cdot {\mathbb Q}_{n}}a_{\lambda}(f)e^{i\lambda t},\quad t\in {\mathbb R}
$$
is semi-anti-periodic. This can be inspected as in the proof of \cite[Proposition 2]{andres1} since the function $f_{N}(\cdot)$ used therein is anti-periodic with the anti-period $\pi q_{1}\cdot \cdot \cdot q_{N}/\theta.$
\end{example}

\begin{example}\label{gaston}
Roughly speaking, it is well known that the unique solution of the heat equation $u_{t}(x,t)=u_{xx}(x,t),$ $x\in {\mathbb R},$ $t\geq 0,$ accompanied with the initial condition $u(x,0)=f(x),$ is given by
$$
u(x,t):=\frac{1}{2\sqrt{\pi t}}\int^{+\infty}_{-\infty}e^{-\frac{(x-s)^{2}}{4t}}f(s)\, ds,\quad x\in {\mathbb R},\ t\geq 0.
$$
By the conclusion from \cite[Example 2.1]{has1}, we know that, if the function $f(\cdot)$ is Bloch $(p,k)$-periodic, then the solution
$u(x,\cdot)$ is likewise Bloch $(p,k)$-periodic ($p>0,$ $k\in {\mathbb R}$). Using this fact, the dominated convergence theorem and Proposition \ref{okje}, we get that, if $f(\cdot)$ is semi-Bloch $k$-periodic, then the solution
$u(x,\cdot)$ will be likewise semi-Bloch $k$-periodic.
\end{example}

It is necessary to note that any of the above introduced function spaces becomes
a metric space when equipped with the metric 
\begin{equation*}
d(f,g):=\sup_{x\in I}\Vert f(x)-g(x)\Vert 
\end{equation*}
inherited from $C_{b}(I:E).$

Let us recall that, for every almost anti-periodic function $f : [0,\infty) \rightarrow E,$ we have that the function ${\mathbb E}f : {\mathbb R} \rightarrow E$ is likewise almost anti-periodic; the same assertion holds for the function $f : [0,\infty) \rightarrow E$
which belongs to the space $ANP_{0}([0,\infty) : E)$ or $ANP([0,\infty) : E);$ see \cite{kosticdaniel}. Now we will state and prove the following proposition:

\begin{prop}\label{danice}
Let $k\in {\mathbb R},$ let $p>0,$ and let a function $f\in C_{b}([0,\infty) : E)$ be given. If $f(\cdot)$ is Bloch $(p,k)$-periodic (semi-Bloch $k$-periodic, semi-anti-periodic), then the function ${\mathbb E}f(\cdot)$ is likewise 
Bloch $(p,k)$-periodic (semi-Bloch $k$-periodic, semi-anti-periodic).
\end{prop}

\begin{proof}
Suppose first that $f(\cdot)$ is Bloch $(p,k)$-periodic. Then $f(x+p)=e^{ikp}f(x),$ $x\geq 0;$ we need to show that $({\mathbb E}f)(x+p)=e^{ikp}({\mathbb E}f)(x),$ $x\in {\mathbb R},$ i.e., $[W(x+p)f](0)=e^{ikp}[W(x)f](0),$ $x\in {\mathbb R}.$
Since $W(x+p)=W(x)W(p),$ $x\in {\mathbb R},$ it suffices to show that $[W(x)f(\cdot +p)](0)=e^{ikp}[W(x)f](0),$ $x\in {\mathbb R},$ i.e., $[W(x)e^{ikp}f(\cdot)](0)=e^{ikp}[W(x)f](0),$ $x\in {\mathbb R},$ which is true. If $f(\cdot)$ is semi-Bloch $k$-periodic,
then Proposition \ref{okje}(iii) yields that there exists a sequence $(f_{n})$ in ${\mathcal B}_{k}( [0,\infty) : E)$ such that $\lim_{n\rightarrow \infty}f_{n}(x)=f(x)$ uniformly in $[0,\infty).$ Due to the supremum formula clarified in \cite[Theorem 2.1.1(xi)]{kostic1}, we have that 
$\lim_{n\rightarrow \infty}({\mathbb E}f_{n})(x)=({\mathbb E}f)(x)$ uniformly in ${\mathbb R}.$ By the first part of proof, we know that for each $n\in {\mathbb N}$ the function $({\mathbb E}f_{n})(\cdot)$ belongs to the space ${\mathcal B}_{k}({\mathbb R} : E)$. Applying again Proposition \ref{okje}(iii), we get that ${\mathbb E}f(\cdot)$ is likewise  semi-Bloch $k$-periodic. The proof is quite similar in the case that $f(\cdot)$ is semi-anti-periodic.
\end{proof}

The proof of following simple proposition is left to the interested reader:

\begin{prop}\label{krew}
Let  $k\in {\mathbb R},$ let $p>0,$ and let $f  : I\rightarrow X$. Then we have:
\begin{itemize}
\item[(i)] If $f(\cdot)$ is  
Bloch $(p,k)$-periodic (semi-Bloch $k$-periodic, semi-anti-periodic), then 
$cf(\cdot)$ is Bloch $(p,k)$-periodic (semi-Bloch $k$-periodic, semi-anti-periodic) for any $c\in {\mathbb C}.$
\item[(ii)] If $X={\mathbb C},$ $\inf_{x\in {\mathbb R}}|f(x)|=m>0$ and
$f(\cdot)$ is  
Bloch $(p,k)$-periodic (semi-Bloch $k$-periodic, semi-anti-periodic), then
$1/f(\cdot)$ is Bloch $(p,-k)$-periodic (semi-Bloch $(-k)$-periodic, semi-anti-periodic).
\end{itemize}
\end{prop}

We close this section by introducing the following definition.

\begin{defn}\label{sqaz}
Let $f\in C_{b}([0,\infty) : E)$ and $k\in {\mathbb R}.$ Then we say that $f(\cdot)$ is asymptotically semi-periodic (semi Bloch $k$-periodic, semi-anti-periodic) iff there exist a function $\phi \in  C_{0}([0,\infty): X)$
and a semi-periodic (semi Bloch $k$-periodic, semi-anti-periodic) function $g : {\mathbb R} \rightarrow E$ such that
$f (t)= g(t)+\phi (t)$ for all $t\geq 0.$
\end{defn}

\section{Stepanov classes of semi Bloch $k$-periodic functions and semi-anti-periodic functions}\label{ret}

As mentioned in the introduction, the notion of Stepanov semi-periodicity has not been analyzed in \cite{andres1}. We will use the following definitions:

\begin{defn}\label{list}
Let $k\in {\mathbb R}$ and $1\leq q<\infty.$ Then we say that
a function $f\in L_{S}^{q}(I:E)$ is  
Stepanov $q$-semi-periodic (Stepanov $q$-semi-Bloch $k$-periodic, Stepanov $q$-semi-anti-periodic), iff the function
$
\hat{f} : I \rightarrow L^{q}([0,1] :E),
$ defined by
\eqref{bruka}, 
is semi-periodic (semi-Bloch $k$-periodic, semi-anti-periodic).
\end{defn}

\begin{defn}\label{list1}
Let $k\in {\mathbb R}$ and $1\leq q<\infty.$ Then we say that
a function $f\in L_{S}^{q}(I:E)$ is  
asymptotically Stepanov $q$-semi-periodic (asymptotically Stepanov $q$-semi-Bloch $k$-periodic, asymptotically Stepanov $q$-semi-anti-periodic), iff the function
$
\hat{f} : I \rightarrow L^{q}([0,1] :E),
$ defined by
\eqref{bruka}, 
is asymptotically semi-periodic (asymptotically semi-Bloch $k$-periodic, asymptotically semi-anti-periodic).
\end{defn}

If a function $f\in L_{S}^{q}(I:E)$ belongs to some of spaces introduced in Definition \ref{list} (Definition \ref{list1}), then $f(\cdot)$ is Stepanov $q$-almost periodic (asymptotically Stepanov $q$-almost periodic). It should be noticed that any of the spaces introduced so far is translation invariant in the sense that if a function $f : I \rightarrow E$ belongs to this space, then for each $\tau \in I$ the function $f(\cdot +\tau)$ likewise belongs to this space.

Let $p>0$ and $k\in {\mathbb R}.$
It should be noted that, if $f : I \rightarrow E$ is Bloch $(p,k)$-periodic (anti-periodic, periodic), then $\hat{f} : I \rightarrow L^{q}([0,1] : E)$ is likewise Bloch $(p,k)$-periodic (anti-periodic, periodic). Further on, it immediately follows from the corresponding definitions that,
if $f : I \rightarrow E$ is semi-Bloch $k$-periodic (semi-anti-periodic, semi-periodic), then $f(\cdot)$ is Stepanov $q$-semi-Bloch $k$-periodic (Stepanov $q$-semi-anti-periodic, Stepanov $q$-anti-periodic) for every number $q\in [1,\infty);$
a large class of non-continuous periodic or Bloch $(p,k)$-periodic functions can be used to provide that the converse statement does not hold in general.
If $1\leq q<q'<\infty$ and $f : I \rightarrow E$
is (asymptotically) Stepanov $q'$-semi-periodic ((asymptotically) Stepanov $q'$-semi-Bloch $k$-periodic, (asymptotically) Stepanov $q'$-semi-anti-periodic), then $f(\cdot)$ is
(asymptotically) Stepanov $q$-semi-periodic ((asymptotically) Stepanov $q$-semi-Bloch $k$-periodic, (asymptotically) Stepanov $q$-semi-anti-periodic). To see that the converse statement does not hold in general, we will provide only one illustrative example:

\begin{example}\label{kule}
Suppose that $1<q<\infty.$
Let us recall that H. Bohr and E. F$\o$lner have constructed an example of a Stepanov $1$-almost periodic function $F: {\mathbb R} \rightarrow {\mathbb R}$ that is not Stepanov $q$-almost periodic (see \cite[p. 70]{bohr-folner}). Moreover, for each $n\in {\mathbb N}$ there exists a bounded periodic function $F_{n} : {\mathbb R} \rightarrow {\mathbb R}$ with at most countable points of discontinuity such that
\begin{align}\label{razm}
\lim_{n\rightarrow \infty}\sup_{t\in {\mathbb R}}\int^{t+1}_{t}\bigl|F_{n}(s)-F(s) \bigr|\, ds=0.
\end{align}
Therefore, $\hat{F_{n}} : {\mathbb R} \rightarrow L^{1}([0,1] : E)$ is a bounded periodic function and, in addition to the above, $\hat{F_{n}}(\cdot)$ is continuous ($n\in {\mathbb N}$). Due to \eqref{razm}, we have that $\lim_{n\rightarrow \infty}\hat{F_{n}}(t)=\hat{F}(t)$ uniformly in $t\in {\mathbb R}.$ This implies that the function $F(\cdot)$ is Stepanov $1$-semi-periodic but not Stepanov $q$-semi-periodic because it is not Stepanov $q$-almost periodic.
\end{example}

\begin{example}\label{pepa-stepa}
Assume that $\alpha, \ \beta \in {\mathbb R}$ and $\alpha \beta^{-1}$ is a well-defined irrational number. Then the functions
$$
f(t)=\sin\Biggl(\frac{1}{2+\cos \alpha t + \cos \beta t}\Biggr),\quad t\in {\mathbb R}
$$
and
$$
g(t)=\cos\Biggl(\frac{1}{2+\cos \alpha t + \cos \beta t}\Biggr),\quad t\in {\mathbb R}
$$
are Stepanov $q$-almost periodic but not almost periodic ($1\leq q<\infty$); see, e.g., \cite{kostic1}. We would like to raise the question whether the functions $f(\cdot)$ and $g(\cdot)$ are Stepanov $q$-semi-periodic for any $1\leq q<\infty ?$
\end{example}

\begin{example}\label{pepa-stepa-levitan}
Define sign$(0):=0.$ Then, for every almost periodic function $f: {\mathbb R} \rightarrow {\mathbb R},$ we have that the function sign$(f(\cdot))$ is Stepanov $q$-almost periodic for any finite real number $q\geq 1$ (see \cite[Example 2.2.2-Example 2.2.3]{kostic1}). Let us consider again the function $f(x):=\sin x +\sin \pi x \sqrt{2},$ $x\in {\mathbb R}.$ Then a simple analysis involving the identity
$f(x)=2\sin x\frac{1+\pi \sqrt{2}}{2}\cos x \frac{\pi \sqrt{2}-1}{2},$ $x\in {\mathbb R}$ shows that the function sign$(f(\cdot))$ is identically equal to a function $F(\cdot)$ of the following, much more general form: Let $(a_{n})_{n\in {\mathbb Z}}$ be a strictly increasing sequence of real numbers satisfying $\lim_{n\rightarrow +\infty}(a_{n+1}-a_{n})=+\infty,$ $\lim_{n \rightarrow +\infty}a_{n}=+\infty$ and 
$\lim_{n \rightarrow -\infty}a_{n}=-\infty.$ Suppose that $(b_{n})_{n\in {\mathbb Z}}$ is a sequence of non-zero real numbers satisfying that the sets $\{n\in {\mathbb Z} : b_{n}<0\}$ and $\{n\in {\mathbb Z} : b_{n}>0\}$ are infinite, as well as that there exists a finite positive constant $c>0$ such that $c\leq |b_{n}-b_{l}|$ for any $n,\ l\in {\mathbb Z}$ such that
$b_{n}b_{l}<0.$
Define the function $F : {\mathbb R} \rightarrow {\mathbb R}$ by $F(x):=b_{n}$ if $x\in [a_{n},a_{n+1}),$ for any $n\in {\mathbb Z}.$
Then $F(\cdot)$ cannot be Stepanov $q$-semi-periodic for any finite real number $q\geq 1.$ Otherwise, for $\epsilon \in (0,c^{q})$ we would be able to find a number $p>0$ such that for each $m\in {\mathbb Z}$ and $x\in {\mathbb R}$ one has:
$$
\int^{1}_{0}\Bigl| F(x+mp+s)-F(x+s)\Bigr|^{q}\, ds <(1/2)^{q}.
$$
Let $n\in {\mathbb Z}$ be such that $[x,x+1]\subseteq [a_{n},a_{n+1})$ and $b_{n}<0,$ say. Without loss of generality, we may assume that the set $\{n\in {\mathbb N} : b_{n}>0\} $ is infinite. Then the contradiction is obvious because, for every sufficiently large numbers $l\in {\mathbb N}$ with $b_{l}>0,$ we can find $m\in {\mathbb N}$
such that $[x+mp,x+mp+1]\subseteq [a_{l},a_{l+1})$ so that
$$
\int^{1}_{0}\Bigl| F(x+mp+s)-F(x+s)\Bigr|^{q}\, ds \geq \bigl| b_{n}-b_{l}\bigr|^{q}\geq c^{q}.
$$
\end{example}

\section{Invariance of semi-Bloch $k$-periodicity and semi-anti-periodicity under the actions of convolution products}\label{jazz}

In this section, we will investigate the invariance of semi-Bloch $k$-periodicity and semi-anti-periodicity under the actions of infinite convolution products and finite convolution products. 

For a given function $g: {\mathbb R} \rightarrow E,$ which is Stepanov $q$-semi-Bloch $k$-periodic or Stepanov $q$-semi-anti-periodic for some number $q\in [1,\infty),$ we will first investigate the qualitative properties of
function $G: {\mathbb R} \rightarrow X,$ given by
\begin{align}\label{wer}
G(t):=\int^{t}_{-\infty}R(t-s)g(s)\, ds,\quad t\in {\mathbb R},
\end{align}
where $(R(t))_{t> 0}\subseteq L(E,X)$ is a strongly continuous operator family satisfying some growth assumptions.
Concerning this question, we have the following result:

\begin{prop}\label{tresbanditos}
Suppose that $k\in {\mathbb R},$ $1\leq q <\infty,$ $1/q' +1/q=1$
and $(R(t))_{t> 0}\subseteq L(E,X)$ is a strongly continuous operator family satisfying that\\ $M:=\sum_{k=0}^{\infty}\|R(\cdot)\|_{L^{q'}[k,k+1]}<\infty .$ If $g : {\mathbb R} \rightarrow E$ is Stepanov $q$-semi-Bloch $k$-periodic (Stepanov $q$-semi-anti-periodic), then the function $G(\cdot),$ given by
\eqref{wer},
is well-defined and semi-Bloch $k$-periodic (semi-anti-periodic).
\end{prop}

\begin{proof}
It can be easily seen that, for every $ t\in {\mathbb R},$ we have $G(t)=\int^{\infty}_{0}R(s)g(t-s)\, ds.$ Due to \cite[Proposition 2.11]{Kos}, we have that $G(\cdot)$ is well-defined and almost periodic.
It remains to be proved that $G(\cdot)$ is semi-Bloch $k$-periodic (semi-anti-periodic). For the sake of brevity, we will consider only  semi-Bloch $k$-periodicity below because the proof for semi-almost-periodicity is analogous (see also \cite[Proposition 3.1]{kosticdaniel}).
Let a number $\epsilon>0$ be given in advance.
Then we can find a finite number $p>0$
such that, for every $m\in {\mathbb Z}$ and $t\in {\mathbb R},$ we have
$$
\int^{t+1}_{t} \bigl\| g(s+mp)-e^{ikmp}g(s) \bigr\|^{q}\, ds \leq \epsilon^{q}, \quad t\in {\mathbb R}.
$$ 
Applying H\"older inequality and this estimate, we get that
\begin{align*}
\bigl\| G(t&+mp) -e^{ikmp}G(t)\bigr\| 
\\ & \leq \int^{\infty}_{0}\|R(r) \|\cdot \bigl\| g(t+mp -r)-e^{ikmp}g(t-r) \bigr\| \, dr
\\ & =\sum _{k=0}^{\infty} \int^{k+1}_{k}\|R(r) \| \cdot \bigl\| g(t+mp -r)-e^{ikmp}g(t-r) \bigr\| \, dr
\\ & \leq \sum _{k=0}^{\infty} \|R(\cdot)\|_{L^{q'}[k,k+1]}\Biggl(\int^{k+1}_{k}\bigl\| g(t+mp -r)-e^{ikmp}g(t-r) \bigr\|^{q} \, dr\Biggr)^{1/q}
\\ & =\sum _{k=0}^{\infty} \|R(\cdot)\|_{L^{q'}[k,k+1]} \Biggl(\int_{t-k-1}^{t-k}\bigl\| g(mp+s)-e^{ikmp}g(s) \bigr\|^{q} \, ds\Biggr)^{1/q}
\\ & \leq \sum _{k=0}^{\infty} \|R(\cdot)\|_{L^{q'}[k,k+1]} \epsilon = M \epsilon,\quad t\in {\mathbb R},
\end{align*}
which clearly implies the required. 
\end{proof}

The above result can be simply applied in the study of existence and uniqueness of semi-Bloch $k$-periodic (semi-anti-periodic) solutions of fractional Cauchy 
inclusion
\begin{align*}
D_{t,+}^{\gamma}u(t)\in {\mathcal A}u(t)+f(t),\ t\in {\mathbb R},
\end{align*}
where $D_{t,+}^{\gamma}$ denotes the Riemann-Liouville fractional derivative of order $\gamma \in (0,1],$ 
$f : {\mathbb R} \rightarrow E$ satisfies certain properties, and ${\mathcal A}$ is a closed multivalued linear operator (see \cite{kostic1} and \cite{kosticdaniel} for the notion and more details). 

In the following proposition, which can be deduced with the help of Proposition \ref{tresbanditos} and the proof of \cite[Propostion 2.13]{Kos}, we analyze the invariance of asymptotical semi-Bloch $k$-periodicity and asymptotical semi-anti-periodicity under the actions of finite convolution products.

\begin{prop}\label{stewpa-wqer}
Suppose that $k\in {\mathbb R},$ $1\leq q <\infty,$ $1/q' +1/q=1$
and $(R(t))_{t> 0}\subseteq L(E,X)$ is a strongly continuous operator family satisfying that, for every 
$s\geq 0,$ we have that
$$
m_{s}:=\sum_{k=0}^{\infty}\|R(\cdot)\|_{L^{q'}[s+k,s+k+1]}<\infty .
$$ 
Suppose, further, that 
the Stepanov $q$-bounded functions $g: {\mathbb R} \rightarrow X$ and
$q: [0,\infty)\rightarrow X$ satisfy that $g(\cdot)$ is 
Stepanov $q$-semi-Bloch $k$-periodic (Stepanov $q$-semi-anti-periodic), $\hat{q} \in C_{0}([0,\infty) : L^{q}([0,1]: E))$
and
$f(t)=g(t)+q(t)$ for all $t\geq 0.$
Let there exist a finite number $M >0$ such that
the following holds:
\begin{itemize}
\item[(i)] $ \lim_{t\rightarrow +\infty}\int^{t+1}_{t}\bigl[\int_{M}^{s}\|R(r)\| \|q(s-r)\| \, dr\bigr]^{q}\, ds=0.$ 
\item[(ii)] $\lim_{t\rightarrow +\infty}\int^{t+1}_{t}m_{s}^{q}\, ds=0.$
\end{itemize}
Then the function $H(\cdot),$ given by
\begin{align*}
H(t):=\int^{t}_{0}R(t-s)f(s)\, ds,\quad t\geq 0,
\end{align*}
is well-defined, bounded and asymptotically $S^{q}$-semi-Bloch $k$-periodic (asymptotically $S^{q}$-semi-anti-periodic). 
\end{prop}

The above result can be applied in the qualitative analysis of  asymptotically (Stepanov $q$-)semi-Bloch $k$-periodic solutions and asymptotically (Stepanov $q$-)semi-anti-periodic solutions of the following abstract Cauchy inclusion
\[
\hbox{(DFP)}_{f,\gamma} : \left\{
\begin{array}{l}
{\mathbf D}_{t}^{\gamma}u(t)\in {\mathcal A}u(t)+f(t),\ t\geq 0,\\
\quad u(0)=x_{0},
\end{array}
\right.
\]
where ${\mathbf D}_{t}^{\gamma}$ denotes the Caputo fractional derivative of order $\gamma \in (0,1],$ $x_{0}\in E,$
$f : [0,\infty) \rightarrow E$ satisfies certain properties, and ${\mathcal A}$ is a closed multivalued linear operator (see \cite{kostic1} and \cite{kosticdaniel} for the notion and more details). 

Finally, let $B$ be a subset of ${\mathbb{R}}^{s}$ and $f:{\mathbb{R}}\times B\rightarrow{\mathbb{R}}^s$. Then we say that the function $f(\cdot)$ is
uniformly semi-Bloch $k$-periodic function iff for any
compact subset $K$ of $B$, we have 
\begin{align*}
\forall\varepsilon>0\quad\exists p\geq0\quad\forall m\in{\mathbb{Z}}\quad
\forall x\in{\mathbb{R}}\quad \forall\alpha\in K\quad
\|f(x+mp,\alpha)-e^{ikmp}f(x,\alpha)\|_{{\mathbb{R}}^s}\leq\varepsilon.
\end{align*}
We close the paper with the observation that we can simply reformulate \cite[Proposition 3]{andres1} for uniformly semi-Bloch $k$-periodic functions (uniformly semi-anti-periodic functions)
and provide certain applications to matrix differential equations, as it has been done in \cite[Theorem 4]{andres1} for semi-periodic functions.


\begin{thebibliography}{99}

\bibitem{andres1} 
J. Andres, D. Pennequin, 
\emph{Semi-periodic
solutions of difference and differential equations,} Bound. Value Probl. 
\textbf{141} (2012), 1--16.

\bibitem{andres2}  
J. Andres, D. Pennequin,  
\emph{Limit-periodic solutions
of difference and differential systems without global Lipschitzianity
restricitons,} 
J. Differ. Equ. Appl. \textbf{24} (2018), 955--975.

\bibitem{Ba}
{H. Bart, S. Goldberg,}
{\it Characterizations of almost periodic strongly continuous groups
and semigroups,}
Math. Ann. \textbf{236} (1978), 105--116.

\bibitem{besi} 
A. S. Besicovitch, 
\emph{Almost Periodic Functions,} 
Dover Publ., New York, 1954.

\bibitem{bohr-folner}
H. Bohr, E. F$\o$lner,
\emph{On some types of functional spaces: A contribution to the theory of almost periodic functions,}
Acta Math. {\bf 76} (1944), 31--155.

\bibitem{Ch}
{Y. Q. Chen,}
{\it Anti-periodic solutions for semilinear evolution equations,}
J. Math. Anal. Appl. {\bf 315} (2006), 337--348.

\bibitem{Di}
T. Diagana,
{\it Almost Automorphic Type and Almost Periodic Type Functions in Abstract Spaces,}
Springer, New York, 2013.

\bibitem{dimba}
W. Dimbour, V. Valmorin, 
{\it Asymptotically antiperiodic solutions for a nonlinear differential equation with piecewise
constant argument in a Banach space,}
Appl. Math. {\bf 7} (2016), 1726--1733.

\bibitem{has2} 
M. F. Hasler, 
\emph{Bloch-periodic generalized functions,}
Novi Sad J. Math. \textbf{46} (2016), 135--143.

\bibitem{has1} 
M. F. Hasler, G.M.N'. Gu\'{e}r\'{e}kata, 
\emph{Bloch-periodic functions and some applications,} Nonlinear Studies \textbf{21
} (2014), 21--30.

\bibitem{henri}
 H. R. Henr\'{\i}quez, 
\emph{On Stepanov-almost periodic
semigroups and cosin functions of operators,} 
J. Math. Anal. Appl. \textbf{146} (1990), 420--433.

\bibitem{guer1} 
G.M.N'. Gu\'er\'ekata, 
\emph{Almost Automorphic and Almost
Periodic Functions in Abstract Spaces,} 
Kluwer Acad. Publ., Dordecht, 2001.

\bibitem{guer2} 
G.M.N'. Gu\'er\'ekata, 
V. Valmorin,
\emph{Antiperiodic
solutions of semilinear integrodifferential equations on a Banach space,}
Appl. Math. Computations \textbf{2018} (2012), 11118--11124.

\bibitem{kostic1} 
M. Kosti\'c, 
\emph{Almost Periodic and Almost Automorphic Type Solutions of Abstract Volterra Integro--Differential Equations,} 
W. de Gruyter, Berlin, 2019.

\bibitem{Kos}
{M. Kosti\'c,}
{\it Existence of generalized almost periodic and asymptotic almost periodic
solutions to abstract Volterra integro-differential equations,}
Electron. J. Differential Equations, vol. 2017, no. {\bf 239} (2017), 1--30.

\bibitem{kostic2} 
M. Kosti\'c, D. Velinov, 
\emph{Asymptotically Bloch-periodic solutions of abstract fractional nonlinear
differential inclusions with piecewise constant argument,} 
Funct. Anal. Appr. Comp. {\bf 9} (2017), 27--36.

\bibitem{kosticdaniel} 
M. Kosti\'c, D. Velinov, 
\emph{A note on almost anti-periodic functions in Banach spaces,}
Kragujevac J. Math. {\bf 44} (2020), 287--297.

\bibitem{Li}
{J. Liu, L. Zhang,}
{\it Existence of anti‑periodic (differentiable)
mild solutions to semilinear differential
equations with nondense domain,}
SpringerPlus (2016) {\bf 5}:704
DOI 10.1186/s40064-016-2315-1.

\bibitem{Liu}
{ J. H. Liu, X. Q. Song, L. T. Zhang,}
{\it  Existence of anti-periodic mild solutions to semilinear nonautonomous evolution equations,}
J. Math. Anal. Appl. {\bf 425} (2015), 295--306.

\end{thebibliography}
\end{document}